\UseRawInputEncoding
\documentclass[12pt]{article}
\usepackage[centertags]{amsmath}
\usepackage{amsfonts}
\usepackage{amssymb}
\usepackage{latexsym}
\usepackage{amsthm}
\usepackage{newlfont}
\usepackage{graphicx}
\usepackage{listings}
\usepackage{booktabs}
\usepackage{abstract}
\usepackage{xcolor}
\date{}

\newlength{\defbaselineskip}
\newcommand{\setlinespacing}[1]%
           {\setlength{\baselineskip}{#1 \defbaselineskip}}

\newcommand{\N}{{\mathbb{N}}}

\newcommand{\actaqed}{\hfill $\actabox$}
{\medskip\noindent \textit{Proof of #1. }}%
{\actaqed \medskip}

\def\D{{\mathcal D}}

\def\cB{{\mathcal B}}
\def\cC{{\mathcal C}}

\def \Tr{\mathcal T}
\def \cK{\mathcal K}

\def \V{\mathcal V}
\def \cW{\mathcal W}
\def \cE{\mathcal E}

\def \cN{\mathcal N}
\def \cL{\mathcal L}

\def\R{{\mathbb R}}
\def\Z{\mathbb Z}

\def \Td{{\mathbb T}^d}
\def \T{\mathbb T}
\def\bP{\mathbb P}
\def\bE{\mathbb E}
\def\bbC{\mathbb C}
\def\Om{\Omega}
\def \<{\langle}
\def\>{\rangle}

\def \e{\varepsilon}
\def \va{\varepsilon}
\def \de{\delta}

\def \ff{\varphi}
\def\la{\lambda}
\def \ro{\varrho}

\def \sp{\operatorname{span}}

\def\bx{\mathbf x}
\def\by{\mathbf y}

\def\bk{\mathbf k}
\def\bv{\mathbf v}
\def\bu{\mathbf u}
\def\bw{\mathbf w}

\def\bs{\mathbf s}
\def\bN{\mathbf N}

\def\bF{\mathbf F}

\def\al{\alpha}
\def\bt{\beta}

\def\La{\Lambda}

\newtheorem{Theorem}{Theorem}[section]
\newtheorem{Lemma}{Lemma}[section]

\newtheorem{Definition}{Definition}[section]
\newtheorem{Proposition}{Proposition}[section]
\newtheorem{Remark}{Remark}[section]

\newtheorem{Corollary}{Corollary}[section]

\numberwithin{equation}{section}

\newcommand{\be}{\begin{equation}}
\newcommand{\ee}{\end{equation}}

\begin{document}

\title{Sampling discretization of the uniform norm}

\author{B. Kashin\thanks{Steklov Institute of Mathematics, Lomonosov Moscow State University, and Moscow Center for Fundamental and Applied Mathematics,},\, S. Konyagin\thanks{Steklov Institute of Mathematics and Lomonosov Moscow State University,},\, and V. Temlyakov\thanks{University of South Carolina, Steklov Institute of Mathematics, Lomonosov Moscow State University, and Moscow Center for Fundamental and Applied Mathematics.} }
\maketitle
\begin{abstract}
{Discretization of the uniform norm of functions from a given finite dimensional subspace of continuous functions is studied.  We pay special attention to the case of trigonometric polynomials with frequencies from an arbitrary finite set with fixed cardinality. We give two different proofs of the fact that for any $N$-dimensional subspace of the space of continuous functions it is sufficient to use $e^{CN}$ sample points for an accurate upper bound for the uniform norm. Previous known results show that one cannot improve on the exponential growth of the number of sampling points for a good
discretization theorem in the uniform norm. 
  Also, we prove a general result, which connects the upper bound on the number of sampling points in the discretization theorem for the uniform norm with the best $m$-term bilinear approximation of the Dirichlet kernel associated with the given subspace. 
We illustrate the application of our technique on the example of trigonometric polynomials.}
\end{abstract}

\section{Introduction}
\label{Int}

 Recently, a systematic study of the problem of discretization of the $L_q$, $1\le q\le\infty$, norms of elements of finite dimensional subspaces has begun (see the survey paper \cite{KKLT}). 
 Discretization of the $L_q$, $1\le q\le \infty$, norms is a vast and rapidly developing area of mathematical analysis. Many results are known in the case, when functions come from specific finite dimensional subspaces -- trigonometric polynomials, algebraic polynomials, defined on different domains, including a sphere, and others. The reader can find a discussion of these results in the very recent survey paper \cite{KKLT}. Here we only cite results directly related to the new results of the paper. In this paper we concentrate on discretization results for arbitrary finite dimensional subspaces. Also, we try to find a characteristic of finite dimensional subspaces which governs the discretization of the uniform norm properties of these subspaces (see Section \ref{A}). We now give very brief historical remarks. 

 The first results in sampling discretization were obtained by Bernstein \cite{Bern1} and \cite{Bern2} (see also \cite{Z}, Ch.10, Theorem (7.28)) for discretization of the uniform norm ($L_\infty$ norm) of the univariate trigonometric polynomials in 1931-1932. 
The first results in sampling discretization of the integral norms $L_q$, $1\le q<\infty$, were obtained by Marcinkiewicz ($1<q<\infty$) and
by Marcinkiewicz-Zygmund ($1\le q\le \infty$) (see \cite{Z}, Ch.10, Theorems (7.5) and (7.28)) for discretization of the $L_q$ norms of the univariate trigonometric polynomials in 1937. Therefore, we also call sampling discretization results  {\it Bernstein-type theorems} in the case of discretization of the uniform norm of functions from a finite dimensional subspace and {\it Marcinkiewicz-type theorems} in the case of integral norms of functions from a finite dimensional subspace (see \cite{VT158}, \cite{VT159}, \cite{DPTT}). In the literature this kind of results is also known under the name {\it Marcinkiewicz-Zygmund inequalities} (see, for instance, \cite{DW}, \cite{Gro} and references therein).
We discuss here the way which uses function values at a fixed finite set of points. This way of discretization is called {\it sampling discretization}. It is known from previous results that the discretization problem in the case $q=\infty$ (uniform norm) 
brings new effects compared to the case $q<\infty$. In this paper we concentrate on the case 
$q=\infty$. We refer the reader to the survey paper \cite{DPTT} and the papers \cite{DPSTT1}, \cite{DPSTT2}, \cite{Kos} for results on sampling discretization of the $L_q$ norm for $q\in [1,\infty)$.
For optimal in the sense of order results on sampling discretization of the $L_2$ norm see \cite{LT}.

We discussed in \cite{VT168} (see also \cite{DPTT}) the following setting of the discretization problem of the uniform norm. 
Let $Q \subset \Z^d$ be a finite set. Denote $N:=|Q|$ its cardinality. Consider a subspace 
$\Tr(Q)$ of trigonometric polynomials of the form
$$
f(\bx) = \sum_{\bk\in Q} c_\bk e^{i(\bk,\bx)}.
$$
Let $ \xi:=\xi(m):=\{\xi^j\}_{j=1}^m \subset \Td$ be a finite set of points. Clearly, 
$$
\|f\|_{\xi(m)} := \max_{1\le j\le m} |f(\xi^j)| \le \|f\|:=\|f\|_\infty
$$
where $\|\cdot\|_\infty$ is the uniform norm.
We are interested in estimating the following quantities
$$
D(Q,m):=D(Q,m,d):= \inf_{\xi(m)}\sup_{f\in\Tr(Q)}\frac{\|f\|}{\|f\|_{\xi(m)}},
$$
$$
D(N,m):=D(N,m,d):= \sup_{Q,|Q|=N} D(Q,m,d).
$$
Certainly, one should assume that $m\ge N$. Then the characteristic $D(Q,m)$ guarantees that there exists a set of $m$ points $\xi(m)$ such that for any $f\in\Tr(Q)$ we have
$$
\|f\|\le D(Q,m)\|f\|_{\xi(m)}.
$$
In the case $d=1$ and $Q=[-n,n]$ the classical Marcinkiewicz theorem (see \cite{VTbookMA}, p. 24)
gives that $D([-n,n],4n)\le C$. Similar relation holds for $D([-n_1,n_1]\times\cdots\times[-n_d,n_d], (4n_1)\times\cdots\times(4n_d))$ (see \cite{VTbookMA}, p. 102). 

It was proved in \cite{VT168} (see also \cite{DPTT}) that for a pair $N$, $m$, such that $m$ is of order $N$ we have that $D(N,m)$ is of order $N^{1/2}$. We formulate this result as a theorem.

\begin{Theorem}[\cite{VT168},\cite{DPTT}]\label{IT1} For any constant $c\ge 1$ there exists a positive constant $C$ such that for any pair of parameters $N$, $m$, with $m\le cN$ we have
$$
D(N,m)\ge CN^{1/2}.
$$
Also, there are two positive absolute constants $c_1$ and $C_1$ with the following property: For any $d\in \N$ we have for $m\ge c_1N$
$$
D(N,m,d)\le C_1N^{1/2}.
$$
\end{Theorem}

The following interesting phenomenon was observed in  \cite{VT168} (see also \cite{DPTT}). It turns out that the inequality $D(N,m) \le C$ is only possible if $m$ grows exponentially with $N$. Namely, the following result was obtained in \cite{VT168} (see also \cite{DPTT}). 

\begin{Theorem}[\cite{VT168},\cite{DPTT}]\label{IT2} Let $\Lambda_N = \{k_j\}_{j=1}^N$ be a lacunary sequence: $k_1=1$, $k_{j+1} \ge bk_j$, $b>1$, $j=1,\dots,N-1$. Assume that a finite set $\xi=\{\xi^\nu\}_{\nu=1}^m\subset \mathbb T$ has
the following property
\begin{equation}\label{I1}
\forall f\in \Tr(\Lambda_N) \qquad \|f\|_\infty \le L\max_{\nu}|f(\xi^\nu)|.
\end{equation}
Then
$$
m \ge (N/e)e^{CN/L^2}
$$
with a constant $C>0$ which may only depend on $b$.
\end{Theorem}

In Section \ref{C} (see Corollary \ref{CC1}) we prove a similar result for randomly chosen sets of frequences. Theorem \ref{IT2} demonstrates that the subspace $ \Tr(\Lambda_N)$ is a difficult one for sampling discretization in the uniform norm. It easily implies the following Proposition \ref{IP1}, which is an extension of the lower bound in Theorem \ref{IT1}.

 \begin{Proposition}\label{IP1} For any number $k\in\N$ and any constant $c\ge 1$ there exists a positive constant $C=C(k,c)$ such that for any pair of parameters $N$, $m$, with $m\le cN^k$ we have
$$
D(N,m)\ge C\left(\frac{N}{\log N}\right)^{1/2}.
$$
\end{Proposition}

In Section \ref{C} we show that a randomly chosen subspace $\Tr(Q)$ 
is also a difficult one for sampling discretization in the uniform norm with high probability. In particular, we give in Section \ref{C} a probabilistic proof of Proposition \ref{IP1}. 
Note that in the case of $\Lambda_N$ the largest frequency $k_N$ is at least of the exponential order in $N$.
In the random example from Section \ref{C} the largest frequency is at most of the polynomial growth in $N$. 

In Section \ref{A} we complement Theorem \ref{IT2} with the corresponding upper bound for an arbitrary $N$-dimensional subspace of the space $\cC(\Om)$ of functions continuous   on a compact set $\Om \subset \R^d$. 

\begin{Theorem}\label{IT3} Let $X_N$ be an $N$-dimensional subspace of $\cC(\Omega)$. 
There exists a set $\xi=\{\xi^\nu\}_{\nu=1}^m$ of $m\le 9^N$ points such that for any $f\in X_N$ we 
have
\be\label{I2}
\|f\|_\infty \le 2\max_\nu |f(\xi^\nu)|.
\ee
\end{Theorem}

Theorem \ref{IT2} shows that Theorem \ref{IT3} provides the best in the sense of order of $\log m$ result for sampling discretization of the uniform norm of $N$-dimensional subspaces. 
 Theorems \ref{IT2} and \ref{IT3} answer the question of the number of sample points needed for the inequality $\|f\|_\infty \le C\max_\nu |f(\xi^\nu)|$ with $C$ being and absolute constant. We now mention some results, where the constant $C$ is allowed to depend on $N$.
The following result was obtained in \cite{NoLN} (see Proposition 1.2.3 there).
\begin{Theorem}[\cite{NoLN}]\label{IT3Nov} Let $X_N$ be an $N$-dimensional subspace of $\cC(\Omega)$ and let $\e>0$. 
There exists a set $\xi=\{\xi^\nu\}_{\nu=1}^N$ of $N$ points such that for any $f\in X_N$ we 
have
\be\label{I2Nov}
\|f\|_\infty \le (N+\e)\max_\nu |f(\xi^\nu)|.
\ee
\end{Theorem}

   The following conditional result, which connects the upper bound in the discretization theorem for the uniform norm with the Nikol'skii-type inequality between $\cC$ and $L_2$ norms, was proved in \cite{DPTT}.

\begin{Theorem}[\cite{DPTT}]\label{IT4}
 Let  $\Omega := [0,1]^d$. Assume that a real $N$-dimensional subspace $Y_N\subset \cC(\Omega)$ satisfies the Nikol'skii-type inequality: for any $f\in Y_N$
 \be\label{I3}
 \|f\|_\infty \le H(N)\|f\|_2,\quad \|f\|_2 := \left(\int_\Omega |f(\bx)|^2d\mu\right)^{1/2},
 \ee
 where $\mu$ is the Lebesgue measure on $\Omega$.
 Then for any $a>1$ there exists a set $\xi(m)=\{\xi^j\}_{j=1}^m\subset \Omega$ with the property:  $m \le a N$ and
 for any $f\in Y_N$ we have  
$$
 \|f\|_\infty \le C(a)H(N)\max_{1\le j\le m} |f(\xi^j)|, 
$$
where $C(a)$ is a  positive constant.

 \end{Theorem}

In Section \ref{A} we prove a result in a style of Theorem \ref{IT4}. We now describe that result. 
Let $\Omega$ be a compact subset of $\R^d$ and let $\mu$ be a probability measure on $\Omega$. Let as above $X_N$ be an $N$-dimensional subspace of $\cC:=\cC(\Omega)$. Assume that there exists an orthonormal basis $\{u_j\}_{j=1}^N$ with respect to the measure $\mu$ of the subspace $X_N$. Denote
$$
\D_{X_N}(x,y) := \sum_{j=1}^N u_j(x)\bar u_j(y), \quad x,y \in \Omega,
$$
the corresponding Dirichlet kernel. Then for any $f\in X_N$ we have
$$
f(x) = \int_{\Omega} \D_{X_N}(x,y)f(y)d\mu.
$$
Consider the following problem of constrained best $M$-term approximation with respect to the bilinear dictionary. Define
\begin{align*}
\cB_M(X_N^\perp):= &\Big\{\cW\,:\, \cW(x,y)= \sum_{i=1}^M w_i(x)v_i(y),\, w_i\in \cC,\,v_i\in L_1, i=1,\dots,M,\\
&\text{satisfying the condition: For any $f\in X_N$ and each $x\in \Omega$ we have}\\
&\int_{\Omega} \cW(x,y) f(y)d\mu =0\Big\}.
\end{align*}
Consider
$$
\sigma_M^c(\D_{X_N})_{(\infty,1)} 
 := \inf_{\cW\in \cB_M(X_N^\perp)} \sup_{x\in \Omega}\|\D_{X_N}(x,\cdot) - \cW(x,\cdot)\|_1.
$$
Under a certain condition (see Condition D below) on the subspace $X_N$ (see Theorem \ref{AT1} below) we prove in Section \ref{A} that there exists a set of points $\{\xi^\nu\}_{\nu=1}^m$ such that for any $f\in X_N$
\be\label{I4}
\|f\|_\infty \le 6\sigma_M^c(\D_{X_N})_{(\infty,1)}\left(\max_{\nu}|f(\xi^\nu)|\right).
\ee
We point out that inequality (\ref{I4}) is a condition result -- it is proved under Condition D, which is 
a condition on the $L_1$ norm discretization of functions from a special subspace related to the subspace $X_N$. Therefore, in order to apply inequality (\ref{I4}) we need to establish the corresponding discretization theorem. Here we use known results on the $L_1$ discretization 
(see Theorem \ref{AT2} below).

In Section \ref{B} we study a problem closely connected with the discussed above problem of
estimation of the quantities  $\sigma_M^c(\D_{X_N})_{(\infty,1)}$ in a special case of trigonometric polynomials. In this case $X_N$ is a subspace $\Tr(Q)$ of trigonometric polynomials with frequencies from $Q\subset \Z^d$, $|Q|=N$. Then $\Omega = [0,2\pi]^d$, $\mu$ is the normalized Lebesgue measure, and for $\bx,\by\in [0,2\pi]^d$ we have
$$
\D_{\Tr(Q)}(\bx,\by) = \D_Q(\bx-\by),\qquad    \D_Q(\bx):= \sum_{\bk\in Q} e^{i(\bk,\bx)}.
$$
Clearly, for any  $g\in \Tr(\La)$, $Q\cap \Lambda =\emptyset$, $|\Lambda|=M$, we have
$$
\sigma_M^c(\D_{\Tr(Q)})_{(\infty,1)} \le \|\D_Q -g\|_1.
$$
This motivates us to construct  generalized de la Vall{\'e}e Poussin kernels for sets $Q$. Namely, for a given $Q\subset \Z^d$ and $M\in\N$ we are interested in construction of an $M$-term trigonometric polynomial $\V_{Q,M}$ such that $\hat \V_{Q,M}(\bk) =1$ for $\bk \in Q$ with small $L_1$ norm. In particular, we prove in Section \ref{B} (see Theorem \ref{BT1}) that for any $Q\subset \Z^d$ there exists $\V_{Q,M}$ such that $\|\V_{Q,M}\|_1 \le 2$ 
provided $M\ge 2^{4|Q|}$. Also, we study there the following question. Find necessary and sufficient conditions on $M$, which guarantee existence of $\V_{Q,M}$ with the property $\|\V_{Q,M}\|_1 \le C_1|Q|^\alpha$, $\alpha \in [0,1/2)$, for all $Q$ of cardinality $n$. These conditions are given in Corollary \ref{BC1} and Remark \ref{BR1}. 
Roughly, they state that $\log M$ should be of order $n^{1-2\alpha}$.

It is known how sampling discretization results for arbitrary subspaces $X_N\subset L_q(\Omega,\mu)$ can be used for proving general inequalities between optimal sampling recovery in $L_q$ and the Kolmogorov widths in the uniform norm of function classes. In the case $q=2$ it was demonstrated in \cite{VT183}. In Section \ref{E} we apply sampling discretization results from Section \ref{A} and obtain general inequalities for optimal sampling recovery in the uniform norm (see, for instance, Theorem \ref{ET2}). 

In Section \ref{D} we give some further comments on sampling discretization results in the uniform norm  and on techniques used for their proofs.

\section{Upper bounds for discretization in the uniform norm}
\label{A}

We begin with a simple proof of the fact that for discretization of the uniform norm of elements of any $N$-dimensional subspace of $\cC(\Omega) $ it is sufficient to use $e^{CN}$
sample points. We now prove Theorem \ref{IT3}.
 
{\bf Proof of Theorem \ref{IT3}.} The proof is based on the idea of covering numbers. 
Let $X$ be a Banach space and let $B_X$ denote the unit ball of $X$ with the center at $0$. Denote by $B_X(y,r)$ a ball with center $y$ and radius $r$: $\{x\in X:\|x-y\|\le r\}$. For a compact set $A$ and a positive number $\e$ we define the covering number $N_\e(A,X)$
 as follows
$$
N_\e := N_\e(A,X) 
:=\min \{n : \exists y^1,\dots,y^n, \, :\,A\subseteq \cup_{j=1}^n B_X(y^j,\e)\}.
$$
It is well known (see, for instance, \cite{VTbook}, p.145) that for any $N$-dimensional Banach space 
$X$ we have
\be\label{A1}
N_\e(B_X,X) \le (1+2/\e)^N.
\ee
We apply bound (\ref{A1}) in the case of $X=X_N$ equipped with the uniform norm $\cC(\Omega)$ with $\e=1/4$ and find a covering net $\cN_{1/4} = \{g^j\}_{j=1}^{N_{1/4}}$, $N_{1/4} \le 9^N$. Let $\xi^\nu$ be a point of maximum of $|g^\nu(x)|$ over $\Omega$. We now prove that the set $\{\xi^\nu\}_{\nu=1}^{N_{1/4}}$ satisfies (\ref{I2}). Clearly, it is sufficient to check (\ref{I2}) for $f\in X_N$, $\|f\|_\infty=1$. By the covering property of the net $\cN_{1/4}$ there exists a $g^{j(f)}\in \cN_{1/4}$ such that
$\|f-g^{j(f)}\|_\infty \le 1/4$. Therefore, we have
$$
|f(\xi^{j(f)})| \ge |g^{j(f)}(\xi^{j(f)})| -1/4 = \|g^{j(f)}\|_\infty-1/4 \ge 1/2,
$$
which proves (\ref{I2}). The proof of Theorem \ref{IT3} is complete.

Note, that it is known (see Theorem \ref{IT2} above and \cite{DPTT}, Theorem 6.7) that even in the case of subspaces $X_N$ spanned by the exponentials $e^{ikx}$, $k\in Q\subset \Z$, $|Q|=N$ the exponential bound in Theorem \ref{IT3} cannot be improved (see Section \ref{B} for details). Therefore, we would like to find some 
characteristics of a subspace $X_N$ which allow us to prove better upper bounds for the number of sampling points for good discretization. 
We now discuss one of such characteristics. Let $\Omega$ be a compact subset of $\R^d$ and let $\mu$ be a probability measure on $\Omega$. Let as above $X_N$ be an $N$-dimensional subspace of $\cC(\Omega)$. We now prove bound (\ref{I4}).

{\bf Proof of  (\ref{I4}).} First, we introduce the condition mentioned right before inequality (\ref{I4}) -- Condition D. Second, we prove Theorem \ref{AT1}, which contains inequality (\ref{I4}).
Let $\cW_M=\sum_{i=1}^M w_i^*v_i^*\in \cB_M(X_N^\perp)$ be such that
$$
\sup_{x\in \Omega}\|\D_{X_N}(x,\cdot) - \cW_M(x,\cdot)\|_1 \le 2\sigma_M^c(\D_{X_N})_{(\infty,1)} .
$$
Assume as above that $\{u_j\}_{j=1}^N$ is an orthonormal basis   with respect to the measure $\mu$ of the subspace $X_N$. Consider a subspace of $L_1(\Omega,\mu)$
$$
Y_S :=\sp \{v_j^*(y)f(y), \bar u_i(y)f(y)\,:\, f\in X_N, \, j=1,\dots,M, \, i=1,\dots,N\}.
$$
Then $S:=\dim Y_S \le (M+N)N$.

{\bf Condition D.} Suppose that $X_N$ is such that there exists a set of points $\{\xi^\nu\}_{\nu=1}^m$ and a set of positive weights $\{\la_\nu\}_{\nu=1}^m$ such that for any $g\in Y_S$
\be\label{A2}
\frac{1}{2} \|g\|_1 \le \sum_{\nu=1}^m \la_\nu |g(\xi^\nu)| \le \frac{3}{2}\|g\|_1 . 
\ee

\begin{Theorem}\label{AT1} Let $X_N\subset \cC(\Omega)$ be an $N$-dimensional subspace. Assume that function $1$ belongs to $X_N$ (if not, we include it, which results in increase of dimension by 1). Assume that $X_N$ satisfies Condition D. Then for the set of points $\{\xi^\nu\}_{\nu=1}^m$ from Condition D we have: For any $f\in X_N$
$$
\|f\|_\infty \le 6\sigma_M^c(\D_{X_N})_{(\infty,1)}\left(\max_{\nu}|f(\xi^\nu)|\right).
$$
\end{Theorem}
\begin{proof} Take any $f\in X_N$. By the properties of the kernels $\D_{X_N}$ and $\cW_M$ we have
for $\cK(x,y):=\D_{X_N}(x,y)-\cW_M(x,y)$
$$
|f(x)| =   |\int_\Omega \cK(x,y)f(y)d\mu|\le \int_\Omega |\cK(x,y)f(y)|d\mu.
$$
It follows directly from the definition of the subspace $Y_S$ that for each $x\in \Omega$ we have $\cK(x,y)f(y) \in Y_S$. Therefore,   we use the left inequality of (\ref{A2}) and continue
$$
\le 2\sum_{\nu=1}^m \la_\nu |\cK(x,\xi^\nu)f(\xi^\nu)| \le \left(\max_{\nu}|f(\xi^\nu)|\right) 2\sum_{\nu=1}^m \la_\nu |\cK(x,\xi^\nu)|.
$$
Our assumption that function $1$ belongs to $X_N$ and the definition of $Y_S$ imply that for each $x\in \Omega$ we have $\cK(x,y) \in Y_S$. Applying the right inequality of (\ref{A2}), we obtain
$$
\le \left(\max_{\nu}|f(\xi^\nu)|\right) 3\|\cK(x,\cdot)\|_1 \le 6\sigma_M^c(\D_{X_N})_{(\infty,1)}\left(\max_{\nu}|f(\xi^\nu)|\right).
$$
Theorem \ref{AT1} is proved.
\end{proof}
This completes the proof of inequality (\ref{I4}).

\begin{Remark}\label{AR1} We can modify Condition D by replacing constants $1/2$ and $3/2$ in (\ref{A2}) by positive constants $c_0$ and $C_0$. Then Theorem \ref{AT1} holds with the constant $6$ replaced by $C(c_0,C_0)$.
\end{Remark}

Theorem \ref{AT1} is a conditional result, which guarantees good discretization of the uniform norm 
under a certain condition on the subspace $X_N$. We now demonstrate that Theorem \ref{AT1} provides a rather powerful tool for the uniform norm discretization. 

First of all we comment on the Condition D. The following result is from \cite{DPSTT2} (see Theorem 2.3 there).

\begin{Theorem}[\cite{DPSTT2}]\label{AT2} Given  $1\leq p< 2$,      an arbitrary  $N$-dimensional subspace $X_N$  of $L_p(\Omega,\mu)$ and any $\va\in (0, 1)$,  there exist   $\xi^1,\ldots, \xi^m\in\Omega$ and   $w_1,\ldots, w_m>0$  such that  $m\leq C_p(\va) N\log^3 N$ with $C_p(\va)$ depending only on $p$, $\va$	 and
	\begin{align}\label{C1a}
	(1-\va)	\|f\|_{p} \leq   \left(\sum_{\nu=1}^m w_\nu |f(\xi^\nu)|^p\right)^{\frac1p}\leq (1+\va) \|f\|_{p},\   \  \forall f\in X_N.
	\end{align}	
\end{Theorem}
We note that there are results on embeddings of finite dimensional subspaces of $L_p$ into $\ell_p^n$, which are related to sampling discretization of the $L_p$ norm (see, for instance, \cite{BLM} and \cite{JS}). 
Professor G. Schechtman kindly pointed out to us that an  improvement  with respect to  the logarithmic factor  of  the estimate  in  Theorem \ref{AT2}   can be obtained using techniques developed for the above embedding problem (see \cite{JS} for those techniques).

An important good feature of Theorem \ref{AT2} is that it applies to any subspace. Using Theorem \ref{AT2} for $p=1$ with $\va=1/2$, we obtain that in Condition D we can guarantee that 
\be\label{A3}
m\le CS(\log S)^3.
\ee
Note, that in some cases we can use special structure of the subspace $Y_S$ in order to improve 
bound (\ref{A3}).  

Second, we demonstrate that Theorem \ref{AT1} combined with (\ref{A3}) and known results on the 
uniform approximation property gives a result similar to Theorem \ref{IT3}. 
There is a general theory of uniform approximation property (UAP), which can be used for proving some estimates for discretization in the uniform norm. We give some definitions from this theory.
For a given subspace $X_N$ of $L_p$, $\dim X_N=N$, and a constant $K>1$, let $k_p(X_N,K)$ be the smallest $k$ such that there is an operator $I_{X_N}: L_p\to L_p$, with $I_{X_N}(f)=f$ for $f\in X_N$, $\|I_{X_N}\|_{L_p\to L_p} \le K$, and rank$( I_{X_N}) \le k$. Define
$$
k_p(N,K):= \sup_{X_N:\dim X_N=N} k_p(X_N,K).
$$
In case $p=\infty$ the $L_\infty$ is the space of continuous functions. The following bound is proved 
in \cite{FJS}
\be\label{A4}
k_\infty(N,K) \le e^{C(K)N}.
\ee
In order to combine bound (\ref{A4}) with Theorem \ref{AT1} we need a representation of an operator of finite rank as an integral operator. For that purpose we show that a general compact set $\Omega$ in the problem of sampling discretization of the uniform norm can always be replaced by a finite set of points. Indeed, let $X_N\subset \cC(\Omega)$ and let $\{u_j\}_{j=1}^N$ be an orthonormal basis 
of $X_N$ with respect to some probability measure $\mu$ on $\Omega$. Take $\e>0$ and using the fact of uniform continuity of functions $u_j$, $j=1,\dots,N$, find $\de>0$ such that for all $x,y\in \Omega$, $\|x-y\|_2\le \de$ we have
$$
|u_j(x)-u_j(y)| \le \e N^{-1/2}.
$$
Let $\cN_\de =\{x^k\}_{k=1}^{N_\de}$ be a $\de$-net of $\Omega$. Then for any $f=\sum_{j=1}^N c_ju_j$ and any $x\in \Omega$ we find $y(x)\in \cN_\de$ such that $\|x-y(x)\|_2\le \de$ and obtain
$$
|f(x)-f(y(x))| \le \sum_{j=1}^N |c_j| \e N^{-1/2} \le \e \|f\|_2 \le \e \|f\|_\infty.
$$
Therefore, for any $f\in X_N$ we have
$$
(1-\e)\|f\|_\infty \le \max_{y\in \cN_\de}|f(y)| \le \|f\|_\infty.
$$
So, we now assume that $\Omega=\{x^j\}_{j=1}^J$. Introduce the measure $\mu(x^j) =1/J$, $j=1,\dots,J$. 
Let $X_N\subset \cC$ be an $N$-dimensional subspace. Assume that function $1$ belongs to $X_N$ (if not, we include it, which results in increase of dimension by 1). The corresponding operator $I_{X_N}$ has rank $M\le e^{C(K)N}$. Represent this operator in the form  
$$
I_{X_N}(f)(x) = \int_\Omega \cK(x,y)f(y)d\mu,\quad \cK(x,y) = \sum_{j=1}^M u_j(x)v_j(y), 
$$
with $ u_j\in \cC$ and  $v_j\in L_1$, $j=1,\dots,M$. Then, 
$$
\sup_x\|\cK(x,\cdot)\|_1 = \|I_{X_N}\|_{\cC \to \cC} \le K, 
$$
which implies that $\sigma_{M+N}^c(\D_{X_N})_{(\infty,1)}\le K$.
Consider a subspace of $L_1$
$$
Y_S := \sp\{v_j(y)f(y)\,:\, f\in X_N, \, j=1,\dots,M\}.
$$
Then $S:=\dim Y_S \le MN$. Thus, by (\ref{A3}) Condition D is satisfied with 
$m\le e^{C'(K)N}$ and Theorem \ref{AT1} gives us an analog of Theorem \ref{IT3}. 

Third, let us consider a known good example of the trigonometric polynomials.  Here, we specify $\Omega := \T^d := [0,2\pi)^d$ and $\mu$ to be the normalized Lebesgue measure $d\mu = (2\pi)^{-d}d\bx$.
By $Q$ we denote a finite subset of $\Z^d$, and $|Q|$ stands for the number of elements in $Q$. Let as above
$$
\Tr(Q):= \left\{f: f=\sum_{\bk\in Q}c_\bk e^{i(\bk,\bx)},\  \  c_{\bk}\in\mathbb{C}\right\}.
$$
 Consider $d$-dimensional parallelepipeds
$$
\Pi(\mathbf N,d) :=\bigl \{\mathbf a\in \Z^d  : |a_j|\le N_j,\
j = 1,\dots,d \bigr\} ,\quad \bN=(N_1,\dots,N_d),
$$
where $N_j$ are nonnegative integers and the corresponding subspaces of the trigonometric polynomials
$$
\Tr(\bN,d) := \Tr(\Pi(\bN,d)).
$$
Then $\dim \Tr(\bN,d) = \vartheta(\mathbf N) := \prod_{j=1}^d (2N_j  + 1)$.
 The following result was obtained by Marcinkiewicz-Zygmund  in the case $d=1$  (see \cite{Z}, Ch.10, \S7). For the multivariate analogs see \cite{VTbookMA}, p.102, Theorem 3.3.15. There are positive constants $C_i(d)$, $i=1,\dots,4$, such that for any $\bN$ there exists a set of points $\xi(m)=\{\xi^j\}_{j=1}^m$, $m\le C_1(d)\vartheta(\mathbf N)$ with the properties: For any $f\in \Tr(\bN,d)$ we have
\be\label{A5}
 C_2(d)\|f\|_1 \le \frac{1}{m} \sum_{j=1}^m |f(\xi^j)| \le C_3(d) \|f\|_1
\ee
and
\be\label{A5'}
\|f\|_\infty \le C_4(d)\|f\|_{\xi(m)}.
\ee

We now demonstrate how relation (\ref{A5'})  can be obtained with a help of Theorem \ref{AT1}. In this case $X_N = \Tr(\bN,d)$ with $N=\vartheta(\mathbf N)$. The Dirichlet kernel is
$$
\D_{X_N}(\bx,\by) =\D_\bN(\bx-\by):= \prod_{j=1}^d \D_{N_j}(x_j-y_j),
$$
where $\bx=(x_1,\dots,x_d)$, $\by=(y_1,\dots,y_d),
$ and $\D_n(t)$ is the classical Dirichlet kernel (see \cite{VTbookMA}, p.6). It is known that the Dirichlet 
kernels $\D_n(t)$ can be extended to the de la Vall{\'e}e Poussin kernels $\V_n(t)$ with the property 
\be\label{A6}
\|\V_n(t)\|_1 \le 3.
\ee
Therefore, extending $\D_\bN(\bx-\by)$ to the 
$$
\V_\bN(\bx-\by):= \prod_{j=1}^d \V_{N_j}(x_j-y_j)
$$
we find that $Y_S \subset \Tr(3\bN,d)$. Using (\ref{A5})  we find that the variant of Condition D with constants $c_0(d)$ and $C_0(d)$ holds for some  $m\le C(d) \vartheta(\mathbf N)$. Applying Theorem \ref{AT1} and Remark \ref{AR1} we conclude the proof of (\ref{A5'}). 

The above two examples -- general subspaces $X_N$ and subspaces of the trigonometric polynomials $\Tr(\bN,d)$ -- show that Theorem \ref{AT1} and the technique of its proof work well in both extreme situations. 

\section{Trigonometric polynomials}
\label{B}

The above example of the subspaces $\Tr(\bN,d)$ shows that the sampling discretization of the uniform norm of elements of these subspaces can be successfully done with the minimal in the sense of order number of sampling points. It turns out that the sampling discretization of the uniform norm of elements of subspaces $\Tr(Q)$ with arbitrary $Q$ is as difficult as the general problem.  

Theorem \ref{IT2} shows that Theorem \ref{IT3} provides the best in the sense of order of $\log m$ result for sampling discretization of the uniform norm of $N$-dimensional subspaces. We now discuss
a connection between sampling discretization of the uniform norm of elements of $\Tr(Q)$ and 
best $M$-term approximation with respect to the trigonometric system $\Tr^d$. For $Q\subset \Z^d$ denote
$$
\D_Q(\bx):= \sum_{\bk\in Q} e^{i(\bk,\bx)}.
$$
We are interested in 
$$
\sigma_M^-(\D_Q)_1 := \inf_{\La\subset \Z^d,\, |\La|=M, \, Q\cap \Lambda =\emptyset}\inf_{c_\bk, \bk\in\Lambda}\left\|\D_Q(\bx)-\sum_{\bk\in\Lambda} c_\bk e^{i(\bk,\bx)}\right\|_1.
$$
Suppose $g_n\in \Tr(\La)$, $Q\cap \Lambda =\emptyset$, $|\La|=M$, $|Q|=N$, is such that
\be\label{A7}
\|\D_Q -g_M\|_1 \le 2\sigma_M^-(\D_Q)_1.
\ee
Then by Theorem \ref{AT1} and by inequality (\ref{A3}) we obtain that there exists a set of points 
$\xi=\{\xi^\nu\}_{\nu=1}^m$ such that for any $f\in \Tr(Q)$ we have
\be\label{A8}
\|f\|_\infty \le 6 \sigma_M^-(\D_Q)_1\max_{\nu}|f(\xi^\nu)|
\ee
with $m$ satisfying
\be\label{A9}
m\le C(d)(N+M)N (\log ((N+M)N))^3.
\ee
We formulate this as a separate statement.

\begin{Proposition}\label{BP1} Let $d\in \N$. There exists a positive constant $C(d)$ such that for 
any $Q \subset \Z^d$, $|Q|=N$, and any $M\ge N$ we have
$$
D(Q,m,d) \le 6\sigma_M^-(\D_Q)_1
$$
provided
$$
m\ge C(d)MN (\log M)^3.
$$
\end{Proposition}

In particular, inequality (\ref{A9}) and Theorem \ref{IT2} imply the following statement for the lacunary set $\La_n$ (see Theorem \ref{IT2} for its definition). Suppose that
$\sigma_M^-(\D_{\La_n})_1 \le C_0$. Then $M\ge e^{C(b,C_0)n}$ with a positive constant $C(b,C_0)$.

We now make some comments on the quantities $D(N,m)$ introduced and discussed in Section \ref{Int}. Theorem \ref{IT1} shows that if for a given $Q\subset \Z^d$ we are satisfied with the bound
$$
\|f\|_\infty   \le C_1|Q|^{1/2} \|f\|_{\xi(m)},\quad f\in \Tr(Q),
$$
then we can find a good set $\xi(m)$ with $m$ of the order $|Q|$. On the example of $Q=\La_n$ (see Theorem \ref{IT2}) we now discuss the following problem. For a given $Q\subset \Z^d$ and $\alpha \in (0,1/2)$, what is the minimal number of points $m$ needed for 
$$
\|f\|_\infty   \le C_0|Q|^{\alpha} \|f\|_{\xi(m)},\quad f\in \Tr(Q)?
$$
Theorem \ref{IT2} with $L=C_0n^{\al}$ gives a lower bound
\be\label{B5}
m \ge \frac{|\Lambda_n|}{e}e^{Cn^{1-2\al}/C_0^2}.
\ee
We now prove an upper bound for $m$. We will use relations (\ref{A8}) and (\ref{A9}). Thus, we need to approximate in the $L_1$ norm the $\D_{\La_n}$. First, we note that in the case $b\ge 3$ the Riesz product
$$
R_{\La_n}(x) := \prod_{j=1}^n (1+\cos(k_jx))
$$
can be used to prove the bound
\be\label{B6}
\sigma_M^-(\D_{\La_n})_1 \le C_2\quad \text{with} \quad M\le 2^{n+1}.
\ee
In the case $b>1$ by splitting $\La_n$ into lacunary sets with big ratio we obtain from (\ref{B6})  \be\label{B7}
\sigma_M^-(\D_{\La_n})_1 \le C_2(b)\quad \text{with} \quad M\le 2^{n+1}.
\ee
Next, for a given $\al\in (0,1/2)$ define $\nu:=[n^{1-\al}]$ and consider blocks
$$
\La^s_n :=\{k_j\}_{j=(s-1)\nu+1}^{s\nu},\quad s=1,\dots,S, \quad S=[n/\nu]+1.
$$
Then, using (\ref{B7}) with $n=\nu$, we obtain for $M=S2^{\nu+1}$
$$
\sigma_M^-(\D_{\La_n})_1 \le \sum_{s=1}^S \sigma_{2^{\nu+1}}^-(\D_{\La^s_n})_1 \le SC_2(b).
$$
This proves that there exists a set $\xi(m)$ with $m$ of order $n^{3-2\al}2^{n^{1-\al}}$ such that
$$
\|f\|_\infty   \le C(b)n^{\alpha} \|f\|_{\xi(m)},\quad f\in \Tr(\La_n).
$$

We now prove an upper bound for a general set $Q$. 

\begin{Theorem}\label{BT1}  For any $d\in\N$ and any $Q\subset \Z^d$ we have for any natural number $s\in [|Q|^{1/2},|Q|]$
$$
\sigma_M^-(\D_Q)_1 \le (e(1+|Q|/s))^{1/2} \quad \text{provided} \quad M\ge |Q|^2 e^{2s}(1+|Q|/s)^{2s}
$$
and in a special case $s=|Q|$ we have a slightly better bound
$$
\sigma_M^-(\D_Q)_1 \le 2 \quad \text{provided} \quad M\ge 2^{4|Q|}.
$$
\end{Theorem} 
\begin{proof}   For two sets $A$ and $B$ from $\Z^d$ define $A+B:= \{a+b\,:\, a\in A,\, b\in B\}$. Denote $Q_1:=Q$ and $Q_{j+1} := Q_j+Q_1$.
Consider the function
$$
f_j(\bx) := \D_{Q_{j+1}}(\bx) \D_{Q_{j}}(-\bx).
$$
We claim that the function $f_j$ has the following representation
\be\label{B8}
f_j(\bx) =|Q_j| \D_{Q}(\bx) + \sum_{\bk\notin Q} c_\bk e^{i(\bk,\bx)}.
\ee
Indeed,
$$
f_j(\bx) = \left(\sum_{\bu\in Q_{j+1}} e^{i(\bu,\bx)}\right)\left(\sum_{\bv\in Q_{j}} e^{-i(\bv,\bx)}\right) =
\sum_{\bv\in Q_j}\sum_{\bv+\bk\in Q_{j+1}}e^{i(\bk,\bx)}.
$$
From the definition of the $Q_{j+1}$ we find that for each $\bk\in Q$ for all $\bv \in Q_j$ we have $\bv+\bk \in Q_{j+1}$. This implies representation (\ref{B8}). Therefore,
\be\label{B9}
\sigma_M^-(\D_Q)_1 \le |Q_j|^{-1}\|f_j\|_1
\ee
for any $M\ge |Q_{j+1}||Q_j|$. Clearly,
\be\label{B10}
\|f_j\|_1 \le \|\D_{Q_{j+1}}\|_2\|\D_{Q_j}\|_2 = (|Q_{j+1}||Q_j|)^{1/2}.
\ee
Therefore, for any $j\in\N$ we have for $M\ge |Q_{j+1}||Q_j|$
\be\label{B11}
\sigma_M^-(\D_Q)_1 \le (|Q_{j+1}|/|Q_j|)^{1/2}.
\ee
We now consider $j\le s$. Then, using notation $n:=|Q|$, we obtain 
\be\label{B11a}
\prod_{j=1}^s |Q_{j+1}|/|Q_j| =|Q_{s+1}|/|Q| \le \binom{n+s}{s+1}n^{-1}.
\ee
We use the known bound for the factorial
$$
(2\pi)^{1/2}r^{r+1/2}e^{-r} \le r! \le er^{r+1/2}e^{-r},
$$ 
which implies that 
\be\label{B12}
\binom{n+s}{s+1}n^{-1} \le e^s\left(1+\frac{n}{s}\right)^s.
\ee
In the case $s=n$ we use a trivial bound $\binom{n+s}{s+1} \le 2^{2n}$.
Therefore, by (\ref{B11a}) there is a $\nu\le s$ such that 
$$
|Q_{\nu+1}|/|Q_\nu| \le e(1+n/s),\quad\text{and in the case $s=n$}\quad |Q_{\nu+1}|/|Q_\nu|\le 4.
$$
This and (\ref{B11}) imply the required upper bounds provided 
$$
M\ge |Q_{s+1}||Q_{s}|,
$$
which is satisfied for $M\ge n^2 e^{2s}(1+ n/s )^{2s}$ due to (\ref{B12}) and for $M\ge 2^{4n}$ in the case  $s=n$. 

\end{proof} 

\begin{Corollary}\label{BC1} There exist two positive absolute constants $C_1$ and $C_2$ such that for any $Q\subset \Z^d$ with $|Q|=n$ we have for $\alpha\in (0,1/2)$ 
$$
\sigma_M^-(\D_Q)_1 \le C_1n^{\alpha} \quad \text{provided} \quad M\ge e^{C_2n^{1-2\alpha}\ln n} .
$$
\end{Corollary}
Note that in the case $\alpha =1/2$ we have $\|\D_Q\|_1\le \|\D_Q\|_2 = |Q|^{1/2}$ and the case $\alpha =0$ is covered by Theorem \ref{BT1}.
\begin{Remark}\label{BR1} There exist two positive absolute constants $c$ and $n_0$ such that for any $\alpha\in [0,1/2)$ we cannot 
replace the condition $M\ge e^{C_2n^{1-2\alpha}\ln n}$ in Corollary \ref{BC1} by the condition 
 $M\ge e^{cn^{1-2\alpha}}$ for $n\ge n_0$.
 \end{Remark}
 \begin{proof} Corollary \ref{BC1} holds for all $Q$ with $|Q|=n$. Therefore, for the proof of Remark \ref{BR1} it is sufficient to consider a particular $Q$.  We take $Q=\Lambda_n$  with $b=3$, defined in Theorem \ref{IT2}, with $n\ge n_0$, where $n_0$ will be specified later. Recall that Theorem \ref{IT2} shows that inequality $D(\Lambda_n,m)\le L$ implies 
 that $m\ge (n/e)e^{Cn/L^2}$ with a positive absolute constant $C$. Suppose $M$ is such that
 $$
 \sigma_M^-(\D_{\Lambda_n})_1 \le C_1n^{\alpha}.
 $$
 Then by Proposition \ref{BP1} we have 
 $$
 D(\Lambda_n,m) \le 6C_1n^{\alpha}\quad \text{for all}\quad m\ge C(1)Mn(\log M)^3.
 $$
 Therefore, we must have
 \be\label{B13}
 C(1)Mn(\log M)^3 \ge (n/e)e^{Cn/L^2}-1\quad \text{with}\quad L=6C_1n^{\alpha}.
 \ee
 Clearly, there are a small enough absolute constant $c>0$ and a large enough $n_0$ such that (\ref{B13}) does not hold for $n\ge n_0$
 if $M\le e^{cn^{1-2\alpha}}+1$. This completes the proof. 
 
 \end{proof}
 
 We now make an observation on the construction of functions $f_j$ in the proof of Theorem \ref{BT1}.
 Remark \ref{BR1} shows that the corresponding construction provides almost optimal result for general sets $Q$. We point out that in a special case, when $d=1$ and $Q=[-n,n]$ we have 
 $$
 f_2(x) =\D_{[-n,n]}(x)\D_{[-2n,2n]}(x) = (2n+1)\V_{n,3n+1}(x),
 $$
 where
 $$
 \V_{n,3n+1}(x) := (2n+1)^{-1}\sum_{k=n}^{3n}\D_{[-k,k]}(x)
 $$
is the de la Vall{\'e}e Poussin kernel. 

It is clear from the definition of functions $f_j$ that the nonzero coefficients $c_\bk$ in the representation (\ref{B8}) satisfy the following inequalities $1\le c_\bk \le |Q_j|$ for all $\bk$.

\section{Subspaces of trigonometric polynomials generated by random spectrum}
\label{C}

In this section we consider subspaces $\Tr(Q)$ with $Q$ being a randomly chosen subset of 
$[1,K]\cap \N$. We begin with some notations. It will be convenient for us to consider the following functions 
\be\label{C1}
\D_Q(x) := \sum_{k\in Q} e^{ikx} =\D_Q^c(x) +i\D_Q^s(x)
\ee
where
$$
\D_Q^c(x) := \sum_{k\in Q} \cos(kx),\qquad \D_Q^s(x) := \sum_{k\in Q} \sin(kx).
$$
For $Q\subset [1,K]\cap \N$ denote $Q^- := ([1,K]\setminus Q)\cap \N$. Let $a\in (0,1)$. Define the functions
$$
 \D_{Q,a}^c :=  \D_Q^c -  a\D_{Q^-}^c,\qquad  \D_{Q,a}^s :=  \D_Q^s -  a\D_{Q^-}^s.
 $$
 Using a brief notation $\D_K$ for $\D_{[1,K]\cap \N}$ we easily obtain that
 \be\label{C2}
 \D_Q^c = (1+a)^{-1}(\D_{Q,a}^c+ a\D_K^c), \quad \D_Q^s = (1+a)^{-1}(\D_{Q,a}^s+a\D_K^s).
 \ee
 
 We now proceed to the random construction. This construction is based on independent random variables with certain properties. It will be convenient for us to present the construction for specific random variables. Let $p\in [0,1/4]$ and let $\eta_p$ be a real function defined on $[0,1]$ as follows: $\eta_p(w)=1$ for $w\in [0,p]$ and $\eta_p(w) = -a(p)$, with $a(p):= p(1-p)^{-1}$, for $w\in (p,1]$. We interpret $\eta_p$ as a random variable defined on $[0,1]$ equipped with the Lebesgue measure $\la$. Then
 \be\label{C3}
 \bE(\eta_p) = 0,\qquad \sigma^2(\eta_p) = \bE(\eta^2) = a(p).
 \ee
 For a $\bw = (w_1,\dots,w_K) \in [0,1]^K$ define
 \be\label{C4}
 Q(\bw) := \{k\,:\, \eta_p(w_k)=1\}.
 \ee
 Then cardinality $|Q(\bw)|$ of the set $Q(\bw)$ is a random variable defined on $[0,1]^K$ equipped with the Lebesgue measure $\la^K := \la\times\cdots\times\la$. Clearly,
 \be\label{C5}
 |Q(\bw)| = \sum_{k=1}^K \chi_{[0,p]}(w_k)
 \ee
 where $\chi_A(w)$ is the characteristic function of the set $A$. Then
 \be\label{C6}
 \bE(|Q(\bw)|) = pK,\quad \bE((|Q(\bw)|-pK)^2) = p(1-p)K.
 \ee
 By Markov's inequality we obtain from (\ref{C6}) that
 \be\label{C7}
 \la^K(\{\bw\,:\, pK/2\le |Q(\bw)| \le 3pK/2\}) \ge 1- \frac{4}{pK}.
 \ee
 We now define two random functions
 $$
 \D^c(x,p,\bw) := \sum_{k=1}^K \eta_p(w_k)\cos(kx), \quad  \D^s(x,p,\bw) := \sum_{k=1}^K \eta_p(w_k)\sin(kx)
 $$
 with $\bw \in [0,1]^K$. 
 \begin{Lemma}\label{CL1} Assume that $N\in \N$ and $K\in \N$ satisfy the inequalities \newline $\log K \le N \le K/4$. Set $p=N/K$. There exist a positive absolute constant $C_0$ and a set $W_0\subset [0,1]^K$ such that 
 $\la^K(W_0) \ge 1-(4K)^{-1}$ and for all $\bw \in W_0$ we have
 \be\label{C7'}
 \|\D^c(x,p,\bw)\|_\infty \le C_0(N\log K)^{1/2},\quad \|\D^s(x,p,\bw)\|_\infty \le C_0(N\log K)^{1/2}.
 \ee
 \end{Lemma}
 \begin{proof} Define $x^l:= \frac{l\pi}{2K}$, $l=1,\dots,4K$. Then, it is known (see, for instance,
 \cite{VTbookMA}, p.102) that for any $f\in \Tr([-K,K]\cap\N)$ we have
 \be\label{C8}
 \|f\|_\infty \le C_1\max_l|f(x^l)|.
 \ee
 For each $l$ by the Bernstein concentration measure inequality (see, for instance, \cite{VTbook}, p.198) we obtain the following bound for the function $\D^c(x^l,p,\bw)$
 \be\label{C9} 
 \la^K(\{\bw\,:\, |\D^c(x^l,p,\bw)|\ge y\}) \le 2\exp\left(-\frac{y^2}{2(a(p)K+y/3)}\right)
 \ee
 and the same bound for the function $\D^s(x^l,p,\bw)$. Then, for $y\ge C_2(N\log K)^{1/2}$ with large enough constant $C_2$ the right side in (\ref{C9}) is less than $(4K)^{-2}/2$. Therefore, for such $y$ we have
 \be\label{C10}
  \la^K(\{\bw\,:\, \max_l\max\{|\D^c(x^l,p,\bw)|,|\D^s(x^l,p,\bw)|\}\ge y\}) \le (4K)^{-1}.
  \ee
 Combining (\ref{C10}) with (\ref{C8}) we complete the proof of Lemma \ref{CL1}.
 \end{proof}
 
 \begin{Theorem}\label{CT1} Assume that $N\in \N$ and $K\in \N$ satisfy the inequalities $\log K \le N \le K/4$ and $N\ge 16$. Set $p=N/K$. There exist a positive absolute constant $c$ and a set $W_1\subset [0,1]^K$ such that $\la^K(W_1) \ge 1/2$ and for each $\bw \in W_1$ we have $N/2\le |Q(\bw)| \le 3N/2$ and 
 \be\label{C10a}
 D(Q(\bw), m) \ge c(N/\log K)^{1/2}
 \ee
 for $ m\le K(\log K)^{1/2}N^{-1/2}$.
 \end{Theorem}
 
 \begin{proof} We specify $p=N/K$ and assume that $N\ge 16$. Then by (\ref{C7}) and Lemma \ref{CL1} we find a set $W_1\subset [0,1]^K$ such that $\la^K(W_1) \ge 1/2$ and for $\bw \in W_1$ we have
 \be\label{C11}
   N/2\le |Q(\bw)| \le 3N/2 
 \ee
 and inequalities (\ref{C7'}) are satisfied. Let $\xi(m)=\{\xi^j\}_{j=1}^m\subset \T$ be given. Set $\xi^{m+1} := \xi^1 +2\pi$ and denote by $x^*$ the center of the largest interval $[\xi^j,\xi^{j+1}]$, $j=1,\dots,m$. Then $|\xi^j-x^*|\ge \pi/m$, $j=1,\dots,m+1$. 
 Take any $\bw\in W_1$ and define a function $f\in \Tr(Q(\bw))$ as follows
 $$
 f(x) := \D_{Q(\bw)}(x-x^*).
 $$
 Then on one hand 
 \be\label{C13}
 f(x^*) = |Q(\bw)| \ge N/2
 \ee
 and on the other hand by (\ref{C2}) and (\ref{C7'}) for any $j$ we have
 \be\label{C14}
 |f(\xi^j)| \le (1+a(p))^{-1}(\|\D^c(x,p,\bw)\|_\infty+\|\D^s(x,p,\bw)\|_\infty +a(p)|\D_K(\xi^j-x^*)|).
 \ee
 Using well known inequality $|\D_K(x)| \le C_3|x|^{-1}$, $|x|\le\pi$, we obtain from (\ref{C14})
 \be\label{C15}
 |f(\xi^j)| \le C_4((N\log K)^{1/2} + Nm/K).
 \ee
 Using our assumption $m\le K(\log K)^{1/2}N^{-1/2}$ we get from (\ref{C15})
  \be\label{C16}
 |f(\xi^j)| \le C_5(N\log K)^{1/2},\quad j=1,\dots,m.
 \ee
 Combining (\ref{C13}) and (\ref{C16}) we complete the proof of Theorem \ref{CT1}.
 
 \end{proof}
 
 \begin{Corollary}\label{CC1} Assume that the conditions of Theorem \ref{CT1} are satisfied. 
 Take $c_1>1$ and assume that the following additional inequality for $N$ holds
 $$
 \max(1,(c_1/c)^2) \le N/\log K \le c_2,
 $$ 
 where $c$ is a constant from (\ref{C10a}). Then for each $\bw \in W_1$ the inequality $D(Q(\bw),m) < c_1$ implies 
 $$
 m> c_2^{-1/2}K \ge c_2^{-1/2} e^{N/c_2} \ge c_2^{-1/2} e^{2|Q(\bw)|/(3c_2)}.
 $$
 \end{Corollary}
 
 Note that in both Theorem \ref{CT1} and Corollary \ref{CC1} we impose the restriction $N\ge \log K$.
 This restriction can be weakened to $N\ge \gamma \log K$, $\gamma >0$, which will result in dependence of constants on $\gamma$. The reader can find a discussion of the case $N\le \gamma \log K$ in Subsection 5.2.
 
 {\bf Proof of Proposition \ref{IP1}.} We prove Proposition \ref{IP1} with $L$ instead of $N$. Assume without loss of generality that $L$ is divisible by 3 and use Theorem \ref{CT1} with $N=2L/3$. Choose $K:= [c(3N/2)^kN^{1/2}] +1$. Then by the assumption of Proposition \ref{IP1} we obtain
 $$
 m\le cL^k \le KN^{-1/2} \le K(\log K)^{1/2}N^{-1/2}
 $$ 
 what is required in Theorem \ref{CT1}. By Theorem \ref{CT1} we obtain 
 $$
 D(L,m)=D(3N/2,m) \ge D(Q(\bw),m) \ge C' \left(\frac{L}{\log L}\right)^{1/2}.
 $$
 
 The above proof of Proposition \ref{IP1} and Proposition \ref{BP1} imply the following statement.
 
 \begin{Proposition}\label{CP1} Suppose that the conditions of Theorem \ref{CT1} are satisfied with 
 $K=N^a$ where $a\in \N$, $a\ge 3$. Then there exist  positive constants $c(d,a)$  and $C(d,a)$ such that 
for all $M\le C(d,a)N^{a-2}$ we have for $\bw \in W_1$
$$
\sigma_M^-(\D_{Q(\bw)})_1 \ge c(d,a) \left(\frac{N}{\log N}\right)^{1/2}.
$$
\end{Proposition}

Note that $\|\D_{Q(\bw)})\|_2 = |Q(\bw)|^{1/2} \le (3N/2)^{1/2}$.

{\bf Comment.} Theorem \ref{AT1} and Proposition \ref{BP1} indicate importance of the best $M$-term approximation characteristics $\sigma_M^c$ and $\sigma_M^-$ for the Bernstein-type discretization theorems. It is an interesting and non-trivial problem to estimate these approximation characteristics. 
We demonstrated above different kinds of techniques (a general theory of uniform approximation property, 
the de la Vall{\'e}e Poussin kernels, the Riesz products, and generalized de la Vall{\'e}e Poussin kernels in the  proof of Theorem \ref{BT1}), which can be successfully applied for proving 
the upper bounds for the $\sigma_M^c$ and $\sigma_M^-$. However, the problem of finding convenient properties of a subspace $X_N$ or a set $Q$ for estimating the $\sigma_M^c$ and $\sigma_M^-$ is an important open problem.

\section{Sampling recovery in the uniform norm}
\label{E}

The problem of recovery (reconstruction) of an unknown function defined on a subset of  $\R^d$ from its samples at a finite number of points is a fundamental problem of pure and applied mathematics. We would like to construct recovering operators (algorithms) which are good in the sense of accuracy, stability, and computational complexity. In this section we discuss 
the issues of accuracy and stability. Following a standard in approximation theory approach we define some 
optimal characteristics -- the Kolmogorov widths and errors of optimal recovery -- for a given function class and establish relations between them. In this section we study sampling recovery in the uniform norm.  This study goes along the lines of the sampling recovery of the $L_2$ norm presented in \cite{VT183}. 
Analysis in \cite{VT183} is based on recent deep results in discretization of the $L_2$ norms of
functions from finite dimensional subspaces  (see \cite{VT158}, \cite{DPSTT2}, and \cite{LT}). 
Our analysis in this section is based on the discretization result -- Theorem \ref{IT3}.
 In this section we consider the sampling recovery on an arbitrary compact subset $\bF$ of $\cC(\Omega)$. The reader can find some recovery results for the uniform norm in the case, when $\bF$ is the unit ball of a reproducing kernel Hilbert space, in the very recent paper \cite{PT}.   
 
Recall the setting 
 of the optimal sampling recovery. For a fixed $m$ and a set of points  $\xi:=\{\xi^j\}_{j=1}^m\subset \Omega$, let $\Phi$ be a linear operator from $\bbC^m$ into $L_p(\Omega,\mu)$.
Denote for a class $\bF$ (usually, centrally symmetric and compact subset of $L_p(\Omega,\mu)$)
$$
\varrho_m(\bF,L_p) := \inf_{\text{linear}\, \Phi; \,\xi} \sup_{f\in \bF} \|f-\Phi(f(\xi^1),\dots,f(\xi^m))\|_p.
$$
The above described recovery procedure is a linear procedure. 
The following modification of the above recovery procedure is also of interest. We now allow any mapping $\Phi : \bbC^m \to X_N \subset L_p(\Omega,\mu)$ where $X_N$ is a linear subspace of dimension $N\le m$ and define
$$
\varrho_m^*(\bF,L_p) := \inf_{\Phi; \xi; X_N, N\le m} \sup_{f\in \bF}\|f-\Phi(f(\xi^1),\dots,f(\xi^m))\|_p.
$$

In both of the above cases we build an approximant, which comes from a linear subspace of dimension at most $m$. 
It is natural to compare quantities $\varrho_m(\bF,L_p)$ and $\varrho_m^*(\bF,L_p)$ with the 
Kolmogorov widths. Let $\bF\subset L_p$ be a centrally symmetric compact. The quantities  
$$
d_n (\bF, L_p) := \operatornamewithlimits{inf}_{\{u_i\}_{i=1}^n\subset L_p}
\sup_{f\in \bF}
\operatornamewithlimits{inf}_{c_i} \left \| f - \sum_{i=1}^{n}
c_i u_i \right\|_p, \quad n = 1, 2, \dots,
$$
are called the {\it Kolmogorov widths} of $\bF$ in $L_p$. In the definition of
the Kolmogorov widths we take for $f\in \bF$, as an approximating element
from $U := \sp \{u_i \}_{i=1}^n$ the element of best
approximation. This means
that in general (i.e. if $p\neq 2$) this method of approximation is not linear.

 Recently an outstanding progress has been done in the sampling recovery in the $L_2$ norm (see, for instance, \cite{CM}, \cite{KU}, \cite{KU2}, \cite{KKP},  \cite{NSU}, \cite{CoDo}, \cite{KUV}, \cite{TU1}). We give a very brief comments on those interesting results. For special sets $\bF$ (in the reproducing kernel Hilbert space setting) the following inequality is known (see \cite{NSU} and \cite{KU}):
\be\label{R2}
\ro_{n}(\bF,L_2) \le C\left(\frac{\log n}{n}\sum_{k\ge cn} d_k (\bF, L_2)^2\right)^{1/2}
\ee
with absolute constants $C,c>0$.  We refer the reader for further results in this direction to the   paper \cite{KU2}. The above results were proved with the aid of a classical type of algorithm -- weighted least squares. A variant of this method -- the empirical risk minimization -- is a standard tool in learning theory and nonparametric statistics (see, for instance, \cite{GKKW} and \cite{VTbook}, Ch.4). 
The weighted least squares algorithm was successfully applied in the sampling recovery for the first time in \cite{CM}. The breakthrough results in sampling discretization of the $L_2$ norm (see \cite{VT158}, \cite{LT}) and in sampling recovery in the $L_2$ norm (see \cite{KU}, \cite{KU2}, \cite{NSU}) are based on results by A.~Marcus, D.A.~Spielman, and
N.~Srivastava from \cite{MSS} (see Corollary 1.5 with $r=2$ there) obtained for solving the Kadison-Singer problem. Also, results from \cite{BSS} play a fundamental role in sampling discretization of the $L_2$ norm. The approach, based 
on \cite{MSS} allows us to obtain optimal (in the sense of order) results for discretization of the $L_2$ norm (see \cite{VT158} and \cite{LT}). For the first time it was done in \cite{VT158} with the help of a lemma from \cite{NOU}. The corresponding lemma from \cite{NOU}  was further generalized in \cite{LT} for proving optimal in the sense of order sampling discretization results. A version of the corresponding lemma from \cite{LT} was used in \cite{NSU} for the sampling recovery.  The first application of the results from \cite{BSS} in the sampling discretization of the $L_2$ norm was done in \cite{VT159}.
The reader can find a detailed discussion of these results in \cite{KKLT}, Section 2.6.

 The following general inequality was proved in \cite{VT183}.
\begin{Theorem}[\cite{VT183}]\label{ET1}  There exist two positive absolute constants $b$ and $B$ such that for any   compact subset $\Omega$  of $\R^d$, any probability measure $\mu$ on it, and any compact subset $\bF$ of $\cC(\Omega)$ we have
$$
\ro_{bn}(\bF,L_2(\Omega,\mu)) \le Bd_n(\bF,L_\infty).
$$
\end{Theorem}

 Very recently Theorem \ref{ET1} was applied to estimate the error of optimal sampling recovery in $L_2$ of functions with small mixed smoothness (see \cite{TU1}). In the case of small smoothness we cannot use (\ref{R2}) because the series in (\ref{R2}) diverges.

In this section we prove the following analog of Theorem \ref{ET1}. 
\begin{Theorem}\label{ET2}  Let  $\Omega$ be a  compact subset   of $\R^d$. For  any compact subset $\bF$ of $\cC(\Omega)$ we have
$$
\ro_{9^n}^*(\bF,L_\infty(\Omega)) \le 5d_n(\bF,L_\infty).
$$
\end{Theorem}

 It is interesting to compare Theorem \ref{ET2} with the following result of E. Novak (see \cite{NoLN}, Proposition 1.2.5)
\be\label{Novak}
\ro_{n}^*(\bF,L_\infty(\Omega)) \le (n+1)d_n(\bF,L_\infty).
\ee
Theorem \ref{ET2} complements inequality (\ref{Novak}) in the case, when the sequence $\{d_n(\bF,L_\infty)\}$ decays 
slow, slower than $1/n$. 
\begin{proof}[Proof of Theorem \ref{ET2}] The proof goes along the lines of the proof of Theorem \ref{ET1}. Recall some notations 
from \cite{VT183} and formulate a result from \cite{VT183} that we use in the proof. Let $X_N$ be an $N$-dimensional subspace of the space of continuous functions $\cC(\Omega)$. For a fixed $m$ and a set of points  $\xi:=\{\xi^\nu\}_{\nu=1}^m\subset \Omega$ we associate with a function $f\in \cC(\Omega)$ a vector
$$
S(f,\xi) := (f(\xi^1),\dots,f(\xi^m)) \in \bbC^m.
$$
Denote
$$
\|S(f,\xi)\|_\infty := \max_{\nu}|f(\xi^\nu)|.
$$
Define the best approximation of $f\in L_\infty(\Omega)$  by elements of $X_N$ as follows
$$
d(f,X_N)_\infty := \inf_{u\in X_N} \|f-u\|_\infty.
$$

We  need the following assumption.

{\bf A1. Discretization.}  Suppose that $\xi:=\{\xi^j\}_{j=1}^m\subset \Omega$ is such that for any 
$u\in X_N$ we have
$$
C_1\|u\|_\infty \le \|S(u,\xi)\|_{\infty}  
$$
with a positive constant $C_1$ which may depend on $d$. 
 
Consider the following well known recovery operator (algorithm)  
$$
\ell \infty(\xi)(f) := \ell \infty(\xi,X_N)(f):=\text{arg}\min_{u\in X_N} \|S(f-u,\xi)\|_\infty.
$$
Note that the above algorithm $\ell \infty(\xi)$ only uses the function values $f(\xi^\nu)$, $\nu=1,\dots,m$. The following theorem is from \cite{VT183}.

\begin{Theorem}[\cite{VT183}]\label{ET3} Under assumption {\bf A1} for any $f\in \cC(\Omega)$ we have
$$
\|f-\ell \infty(\xi)(f)\|_\infty \le (2C_1^{-1}  +1)d(f, X_N)_\infty.
$$
\end{Theorem}

Take any $\e>0$ and choose a subspace $X_n^\e$ of dimension $n$ satisfying: For all $f\in \bF$
\be\label{E1}
d(f,X_n^\e)_\infty \le (1+\e)d_n(\bF,L_\infty).
\ee
By Theorem \ref{IT3} the subspace $X_n^\e$ satisfies assumption {\bf A1} with $m=9^n$, $C_1=1/2$, and some set $\xi:=\{\xi^j\}_{j=1}^m\subset \Omega$. Then, by Theorem \ref{ET3} and by (\ref{E1}) we obtain the inequality: For any $f\in \cC(\Omega)$
$$
\|f-\ell \infty(\xi)(f)\|_\infty \le 5d(f, X_n^\e)_\infty \le 5(1+\e)d_n(\bF,L_\infty).
$$
This completes the proof of Theorem \ref{ET2}.

\end{proof}

In the definition of the quantity $\varrho_m^*(\bF,L_p)$
we allow any mapping $\Phi : \bbC^m \to X_N \subset L_p(\Omega,\mu)$ where $X_N$ is a linear subspace of dimension $N\le m$. We now consider an analog of the quantity $\varrho_m^*(\bF,L_\infty)$ with some restrictions imposed on the mapping $\Phi : \bbC^m \to X_N \subset \cC(\Omega)$. For a given subspace $X_N\subset \cC(\Omega)$, a set $\xi\in \Omega^m$, and given parameters $m\in\N$ and $A>0$ define the following class of nonlinear mappings ($A$-stable mappings)
$$
\cN\cL(\xi,X_N,A) := \{\Phi\, : \, \Phi : \bbC^m \to X_N,\, \|\Phi(a\by)\|_\infty = |a|\|\Phi(\by)\|_\infty,\, 
$$
$$
\|\Phi(S(f,\xi))\|_\infty \le A\|S(f,\xi)\|_\infty,\, f\in X_N\}.
$$
Consider the following recovery characteristic
$$
\varrho_m^*(\bF,L_\infty,A) := \inf_{\Phi\in \cN\cL(\xi,X_N,A); \xi; X_N, N\le m} \sup_{f\in \bF}\|f-\Phi(f(\xi^1),\dots,f(\xi^m))\|_\infty.
$$
This characteristic gives the minimal error that can be achieved with an $A$-stable algorithm.

We begin with a simple statement, which connects the characteristic $\varrho_m^*(X_N^\infty,L_\infty,A)$ 
with discretization of the uniform norm on $X_N$. 

\begin{Proposition}\label{EP1} Inequality $\varrho_m^*(X_N^\infty,L_\infty,A) \le 1/2$ implies that there exists a set $\xi(m)\subset \Omega$ of $m$ points such that for any $f\in X_N$ we have
$$
\|f\|_\infty \le 2A\|S(f,\xi(m))\|_\infty.
$$
\end{Proposition}
\begin{proof} Let $\e>0$ and let $\xi^\e \in \Omega^m$ and $\Phi_\e \in \cN\cL(\xi^\e,X_N,A)$ be such that for any $f\in X_N^\infty$ we have
\be\label{E2}
\|f-\Phi_{\e}(S(f,\xi^\e))\|_\infty \le 1/2+\e. 
\ee
Then for $f\in X_N$ such that $\|f\|_\infty =1$ we obtain
$$
1=\|f\|_\infty = \|f-\Phi_{\e}(S(f,\xi^\e))+\Phi_{\e}(S(f,\xi^\e))\|_\infty \le \|\Phi_{\e}(S(f,\xi^\e))\|_\infty+1/2+\e
$$
and, therefore,
$$
(1/2-\e)\|f\|_\infty \le \|\Phi_{\e}(S(f,\xi^\e))\|_\infty \le A\|S(f,\xi^\e))\|_\infty.
$$
This and a  simple compactness argument, used for $\e\to0$, complete the proof of Proposition \ref{EP1}.

\end{proof}

We now make a comment on stability. It is well known (see, for instance, \cite{DPTT}, Proposition 3.1) that $\varrho_N(X_N^\infty,L_\infty) =0$. In other words, for any $N$-dimensional subspace $X_N\subset \cC(\Omega)$ we can find a set $\xi =\{\xi^j\}_{j=1}^N$ of $N$ points such that 
any $f\in X_N$ can be recovered exactly from the vector $S(f,\xi)$ (even by a linear algorithm). 
However, as we will explain momentarily, such a recovery algorithm may be unstable. As an example we take $X_n := \Tr(\La_n)$, where $\La_n$ is a lacunary set from Theorem \ref{IT2}. Suppose that 
$\varrho_m^*(X_n^\infty,L_\infty,A) \le 1/2$ with some parameters $m$ and $A$. Clearly, this assumption is much weaker than the exact recovery by stable algorithm assumption. Then by Proposition \ref{EP1} there exists a set $\xi(m)\subset \Omega$ of $m$ points such that for any $f\in X_n$ we have
$$
\|f\|_\infty \le 2A\|S(f,\xi(m))\|_\infty.
$$
We apply Theorem \ref{IT2} and obtain that $m\ge (n/e)e^{Cn/(2A)^2}$. This means that for a stable nontrivial approximate recovery on $X_n$ we need exponentially (in $n$) many points. 

We now proceed to a positive result on stable recovery.  
\begin{Proposition}\label{EP2}  Let  $\Omega$ be a  compact subset   of $\R^d$. For  any $N$-dimensional subspace $X_N\subset \cC(\Omega)$ we have
$$
\ro^*_{9^N}(X_N^\infty,L_\infty(\Omega),2) =0.
$$
\end{Proposition}
\begin{proof} We use the construction from 
  the proof of Theorem \ref{IT3}.  Let
 $\cE :=\cN_{1/4}= \{g^j\}_{j=1}^m$, $ m=N_{1/4}$, and let  $\xi(m) = \{\xi^\nu\}_{\nu=1}^m$, $m=N_{1/4}$, where $\xi^\nu$ is a point of maximum of $|g^\nu(\bx)|$ over $\Omega$, be the corresponding sets from the proof of Theorem \ref{IT3}.   Define $\Phi(S(f,\xi(m))):= f$ for $f\in X_N$. 
By Theorem \ref{IT3} we obtain that if for two functions $f,g\in X_N$ we have $S(f,\xi(m))=S(g,\xi(m))$ then $f=g$. Thus, the mapping $\Phi$ is defined in a correct way. Clearly, $f-\Phi(S(f,\xi(m)))=0$.
  It remains to prove that $\Phi \in \cN\cL(\xi(m),X_N,2)$. It follows directly from Theorem \ref{IT3}.
   This completes the proof.

\end{proof}

\begin{Remark}\label{ER1} The above simple argument gives the following conditional result. Suppose that a sunspace $X_N$ and a set of points $\xi=\{\xi^j\}_{j=1}^m$ satisfy condition {\bf A1}. 
Then we have 
$$
\ro^*_{m}(X_N^\infty,L_\infty(\Omega),C_1^{-1}) =0.
$$
\end{Remark}

\section{Discussion}
\label{D}

In this paper we have obtained some results on the sampling discretization of the uniform norm
both for a special case of trigonometric polynomials and for a rather general case of finite dimensional subspaces. We now present a discussion of known results on the sampling discretization of the uniform norm and give some comments on the results obtained above.
 
Let us begin our discussion of known results with the Marcinkiewicz-type discretization theorems for the hyperbolic cross trigonometric polynomials. For $\bs\in\Z^d_+$
  we define
$$
\rho (\bs) := \{\bk \in \Z^d : [2^{s_j-1}] \le |k_j| < 2^{s_j}, \quad j=1,\dots,d\}
$$
where $[x]$ denotes the integer part of $x$.
By $Q_n$ denote
the step hyperbolic cross, i.e.,
$$
Q_n := \cup_{\bs:\|\bs\|_1\le n} \rho(\bs).
$$
Then  the corresponding set of the hyperbolic cross polynomials is given by
$$
\Tr(Q_n) := \{f: f=\sum_{\bk\in Q_n} c_\bk e^{i(\bk,\bx)},\  \  c_\bk\in\mathbb{C}\}.
$$
The problem on
obtaining the sharp Marcinkiewicz-type discretization theorems for the hyperbolic cross trigonometric polynomials is not solved yet.
Since $Q_n \subset \Pi(2^n,\dots,2^n)$, from the result mentioned in the Introduction we get
$$
D(Q_n,m)\le C_1(d),\quad \text{provided} \quad m\ge C_2(d)2^{dn},
$$
with appropriate $C_1(d)$ and $C_2(d)$.
   
  The following nontrivial lower bound was obtained in \cite{KT3} -- \cite{KaTe03}.
 \begin{Theorem}[\cite{KT3} -- \cite{KaTe03}]\label{DT1} Let a finite set $\xi(m)=\{\xi^j\}_{j=1}^m\subset \mathbb T^2$ have a property:
 $$
\forall f \in \Tr(Q_n) \qquad \|f\| \le Bn^\al \|f\|_{\xi(m)}
$$
with some $0\le \al <1/2$. Then
$$
m \ge C_1 2^n n e^{C_2B^{-2}n^{1-2\al}},
$$
where $C_1$ and $C_2$ are positive absolute constants.
\end{Theorem}
In particular, Theorem \ref{DT1} with $\alpha=0$ implies that a necessary condition on $m$ for  $D(Q_n,m)\le C_3$ is $m\ge C_1' |Q_n|^{1+c}$ with positive absolute constant $c$.

We now cite a recent result from \cite{DPTT} on $D(Q_n,m)$. Define
$$
\al_d:=\sum_{j=1}^d \frac{1}{j}\qquad\mbox{ and}\qquad
\beta_d:=d-\al_d.
$$  
\begin{Theorem}[\cite{DPTT}]\label{DT2} For each $d\in\N$ and each $n\in\N$ there exists a set $\xi(m)$ of at most $C_d 2^{n{\alpha_d}}n^{\beta_d}$ points in $[0, 2\pi)^d$ such that for all $f\in\mathcal{T}(Q_n)$,
	$$ \|f\| \le C(d) \|f\|_{\xi(m)}.$$
\end{Theorem}

An operator $T_n$ with the following properties was constructed in \cite{T93}.
The operator $T_n$ has the form
$$
T_n(f) = \sum_{j=1}^m f(\xi^j) \psi_j(\bx),\quad m\le c(d)2^n n^{d-1},\quad \psi_j \in \Tr(Q_{n+d})
$$
and
\be\label{DD1}
T_n(f) =f,\quad f\in \Tr(Q_n),
\ee
\be\label{DD2}
\|T_n\|_{\ell^m_\infty\to \cC} \asymp n^{d-1}.
\ee
Points $\{\xi^j\}$ form the sparse grid (Smolyak net) (see \cite{VTbookMA}, p.268).
Properties (\ref{DD1}) and (\ref{DD2}) imply that all $f\in\Tr(Q_n)$ satisfy the discretization inequality (see \cite{KaTe03})
\be\label{DD3}
\|f\|_\infty \le C(d)n^{d-1} \max_{1\le j\le m} |f(\xi^j)|.
\ee
We now discuss some results from \cite{TT} about the relation between discretization and the Remez-type inequality.
\begin{Definition}\label{DefD1} We say that $f$ satisfies the Remez-type inequality with parameters $p$, $\bt$, $R$ (in other words, $RI(p,\bt,R)$ holds) if for any measurable $B\subset \Om$ with measure $|B|\le \bt$
\be\label{DD4}
\|f\|_{L_p(\Om)} \le R\|f\|_{L_p(\Om\setminus B)}.
\ee
\end{Definition}
The following result is from \cite{TT}.
 \begin{Theorem}[\cite{TT}]\label{DT3} Let $f$ be a continuous periodic function on $\T^d$. Assume that there exists a set $\xi(m)=\{\xi^j\}_{j=1}^m \subset \T^d$ such that for all functions $f_\by(\bx):= f(\bx-\by)$, $\by\in\T^d$, we have the discretization inequality
\be\label{DD5}
\|f_\by\|_\infty \le D\max_{1\le j\le m}|f_\by(\xi^j)|.
\ee
Then for any $B$ with $|B|<1/m$ we have 
\be\label{DD6}
\|f\|_{\cC(\Om)} \le D\|f\|_{\cC(\Om\setminus B)}.
\ee
\end{Theorem}

It was proved in \cite{TT} that the following statement is false.

{\it There exist  $\de>0$, $A$, $c$, and $C$ such that for any $f\in \Tr(Q_n)$ and any set $B\subset \T^d$ of measure
$|B| \le (c2^nn^A)^{-1}$ the Remez-type inequality holds
$$
\|f\|_\infty \le Cn^{(d-1)(1-\de)} \sup_{{\mathbf u}\in \T^d\setminus B} |f({\mathbf u})|.
$$}

The above remark and Theorem \ref{DT3} imply the following statement.

\begin{Corollary}\label{DC1} Take any $\de>0$. Suppose that the point set $\xi(m)$ is such that 
for all $f\in \Tr(Q_n)$ we have
$$
\|f\| \le C(d) n^{(d-1)(1-\de)} \|f\|_{\xi(m)}.
$$
Then for any $A$ we have  $m\ge C(A,d,\delta) 2^n n^A$ with a positive constant $C(A,d,\delta)$.
\end{Corollary}

We now show how one can prove weaker than the first part of Theorem \ref{IT1} results using other methods.

{\bf 6.1. Use of the Remez and Nikol'skii inequalities.} Let a finite set $\cK:=\{k_j\}_{j=1}^n$, $k_j\in\N$, be a lacunary set, i.e. for some $b>1$ we have $k_{j+1}/k_j \ge b$, $j=1,\dots, n-1$. Then it is well known that for any $q<\infty$ we have the following lower bound in the 
Nikol'skii inequality
\be\label{D1}
\sup_{f\in \Tr(\cK)}\|f\|_\infty/\|f\|_q \ge C(q,b) n^{1/2}.
\ee
It is proved in \cite{TT} that the Remez-type inequality (see Definition \ref{DefD1}) with parameters $R$ and $\bt$ implies the following Nikol'skii-type inequality 
\be\label{D2}
\|f\|_\infty \le R\bt^{-1/q}\|f\|_q.
\ee
Using (\ref{D1}) with $n=N$ we obtain from (\ref{D2}) and Theorem \ref{DT3} that 
\be\label{D3}
D(N,m) \ge C(q)m^{-1/q}N^{1/2},
\ee
which implies that for a given $C_0$ and $\e>0$ we have for $m\le C_0N$
\be\label{D4}
D(N,m) \ge C(C_0,\e)N^{1/2-\e}.
\ee

{\bf 6.2. Probabilistic approach.} As above, consider a lacunary set $\cK:=\{k_j\}_{j=1}^n$ and the following random polynomials
$$
f(x,t) := \sum_{j=1}^n r_j(t)e^{ik_jx}
$$
where $r_j(t)$ are the Rademacher functions. 
Then, on one hand, it is well known that for each $t\in [0,1] $ we have
\be\label{D5}
\|f(\cdot,t)\|_\infty \ge c(b)n
\ee
with a positive constant $c(b)$. 

On the other hand, using the concentration measure inequality, for instance, the Hoeffding inequality (see \cite{VTbook}, p.197), we obtain for any fixed $x$
\be\label{D6}
\bP\left\{|f(x,\cdot)| \ge n\eta\right\} \le 4\exp(-cn\eta^2)
\ee
with an absolute positive constant $c$. Inequality (\ref{D6}) implies that for any set $\xi(m)=\{\xi^j\}_{j=1}^m$ of $m$ points with $4m\exp(-cn\eta^2)<1$ there exists $t_0\in[0,1]$ such that 
$\|f(\cdot,t_0)\|_{\xi(m)} \le n\eta$. Combining this with (\ref{D5}) we obtain 
\be\label{D7}
D(\cK,m) \ge c(b)/\eta \quad \text{provided}\quad 4m\exp(-cn\eta^2)<1.
\ee
It is easy to derive from here that
\be\label{D8}
D(\cK,m) \ge c' n^{1/2}(\ln m)^{-1/2}\quad \text{and}\quad D(N,m) \ge c' N^{1/2}(\ln m)^{-1/2}.
\ee

In the above proof of (\ref{D8}) we only use the following properties of the trigonometric system:
uniform boundedness and the inequality (\ref{D5}) for polynomials with frequencies from $\cK$. 
We now explain how the above argument can be used in the case of uniformly bounded orthonormal systems.   The following statement was proved in \cite{Ka1} (see also \cite{KS}, Ch.9, Theorem 9). 

\begin{Theorem}[\cite{Ka1}]\label{DTk} Let $\Phi :=\{\ff_j(x)\}_{j=1}^K$, $x\in [0,1]$, be a uniformly bounded real orthonormal system $\|\ff_j\|_\infty \le B$, $j=1,\dots,K$. There exists a subset $J_s \subset [1,K]\cap \N$
with cardinality $|J_s|=s:= \max([(\log_2 K)/6],1)$ such that for any $a_k$ we have
\be\label{D8k}
\sum_{k\in J_s} |a_k| \le 4B\left\|\sum_{k\in J_s} a_k \ff_k\right\|_\infty.
\ee
\end{Theorem}
Theorem \ref{DTk} provides an analog of inequality (\ref{D5}). Thus, Theorem \ref{DTk} and the above
probabilistic argument give the following result. Denote $\Phi(J) := \sp(\ff_j, j\in J)$. 
\begin{Theorem}\label{DTd} Let $\Phi :=\{\ff_j(x)\}_{j=1}^K$, $x\in [0,1]$, be a uniformly bounded real orthonormal system $\|\ff_j\|_\infty \le B$, $j=1,\dots,K$. There is a positive constant $c(B)$ with the following property.  For any natural number $n\le s:= \max([(\log_2 K)/6],1)$
there exists a subset $J_n \subset [1,K]\cap \N$
with cardinality $|J_n|=n$ such that for any $\xi(m)=\{\xi^j\}_{j=1}^m$ we have
\be\label{D8d}
 \sup_{f\in \Phi(J_n)}(\|f\|_\infty/\|f\|_{\xi(m)}) \ge c(B)  (n/\ln m)^{1/2}.
\ee
\end{Theorem}
In particular, Theorem \ref{DTd} implies that it is necessary to have $m\ge e^{c_1(B,D)n}$, $c_1(B,D)>0$, for the discretization inequality
$$
\|f\|_\infty \le D\|f\|_{\xi(m)},\quad f\in \Phi(J_n)
$$
to hold. 

We point out that the proof of Theorem \ref{DTk} from \cite{Ka1} (see also \cite{KS}, Ch.9, Theorem 9) is probabilistic. This means that Theorem \ref{DTd} holds for a randomly chosen subset $J_n$. 

{\bf 6.3. Comment on Theorem \ref{IT4}.} We now prove a somewhat more general result than Theorem \ref{IT4}. We prove a similar result for complex subspaces and for a general probability measure. 
However our new restrictions on the number of sampling points are weaker than in Theorem \ref{IT4}.
\begin{Theorem}\label{DT4}
 Let  $\Omega$ be a compact subset of $\R^d$. Assume that an \newline $N$-dimensional subspace $Y_N\subset \cC(\Omega)$ (real or complex) satisfies the Nikol'skii-type inequality: For any $f\in Y_N$
 \be\label{D9}
 \|f\|_\infty \le H(N)\|f\|_2,\quad \|f\|_2 := \left(\int_\Omega |f(\bx)|^2d\mu\right)^{1/2},
 \ee
 where $\mu$ is the probability measure on $\Omega$.
 Then there are two positive absolute constants $C_1$ and $C_2$ such that there exists a set $\xi(m)=\{\xi^j\}_{j=1}^m\subset \Omega$ with the property:  $m \le C_1 N$ and
 for any $f\in Y_N$ we have  
$$
 \|f\|_\infty \le C_2H(N)\max_{1\le j\le m} |f(\xi^j)|. 
$$
 \end{Theorem}
 \begin{proof} We use a result on discretization in $L_2$ from \cite{LT} (see Theorem 3.3 there), which is a generalization to the complex case of an earlier result from \cite{DPSTT2} established for the real case. 

\begin{Theorem}[\cite{DPSTT2},\cite{LT}]\label{DT5} If $X_N$ is an $N$-dimensional subspace of the complex $L_2(\Omega,\mu)$, then there exist three absolute positive constants $C_1'$, $c_0'$, $C_0'$,  a set of $m\leq   C_1'N$ points $\xi^1,\ldots, \xi^m\in\Omega$, and a set of nonnegative  weights $\lambda_j$, $j=1,\ldots, m$,  such that
\[ c_0'\|f\|_2^2\leq  \sum_{j=1}^m \lambda_j |f(\xi^j)|^2 \leq  C_0' \|f\|_2^2,\  \ \forall f\in X_N.\]
\end{Theorem}

For our application we need to satisfy a certain assumption on the weights. Here is a simple remark from \cite{VT183}.
\begin{Remark}[\cite{VT183}]\label{DR1} Considering a new subspace $X_N' := \{f\,:\, f= g+c, \, g\in X_N,\, c\in \bbC\}$ and applying Theorem \ref{DT5} to 
the $X_N'$ with $f=1$ ($g=0$, $c=1$) we conclude that a version of Theorem \ref{DT5} holds with $m\le C_1'N$ replaced by $m\le C_1'(N+1)$ and with weights satisfying 
$$
\sum_{j=1}^m \lambda_j \le C_0'.
$$
\end{Remark}

Let now $\xi(m)=\{\xi^j\}_{j=1}^m$ be the set of points from Theorem \ref{DT5} and Remark \ref{DR1} with $X_N = Y_N$. Then $m\le C_1N$ and by the Nikol'skii-type inequality assumption we obtain
$$
 \|f\|_\infty \le H(N)\|f\|_2 \le (c_0')^{-1/2}  \left(\sum_{j=1}^m \lambda_j |f(\xi^j)|^2\right)^{1/2} 
 $$
 $$
 \le (c_0')^{-1/2} \left(\sum_{j=1}^m \lambda_j\right)^{1/2} \max_j |f(\xi^j)| \le C_2 \max_j |f(\xi^j)|.
 $$

 \end{proof}

{\bf Acknowledgements.} The authors are grateful to the referees for their useful comments and suggestions.

The work was supported by the Russian Federation Government Grant N{\textsuperscript{\underline{o}}}14.W03.31.0031.


\begin{thebibliography}{9999}

\bibitem{BSS} J. Batson, D.A. Spielman, and N. Srivastava, Twice-Ramanujan Sparsifiers,   {\it SIAM Review}, {\bf 56} (2014), 315--334.

\bibitem{Bern1} S.N. Bernstein, Sur une classe de formules d'interpolation, Izv. AN SSSR, 1931, No. 9, 1151--1161.

 \bibitem{Bern2} S.N. Bernstein, Sur une modification de la formule d'interpolation de Lagrange, Zapiski Khar'kovskogo matem. tovar., 1932,
v. 5, 49--57. 

\bibitem{BLM} J. Bourgain, J. Lindenstrauss and V. Milman, Approximation of zonoids by zonotopes,   {\it Acta Math.}, {\bf 162} (1989), 73--141.

\bibitem{CM} A. Cohen and G. Migliorati, Optimal weighted least-squares methods, {\it SMAI J. Computational Mathematics}, {\bf 3} (2017), 181--203.

\bibitem{CoDo} A. Cohen and M. Dolbeault, Optimal pointwise sampling for $L^2$ approximation,
arXiv:2105.05545v2 [math.NA] 1 Jun 2021.

 \bibitem{DPTT} F. Dai, A. Prymak, V.N. Temlyakov, and  S.U. Tikhonov, Integral norm discretization and related problems,
  {\it Russian Math. Surveys} {\bf 74:4} (2019),   579--630.
 Translation from
{\it Uspekhi Mat. Nauk}  {\bf 74:4(448)}  (2019),	3--58; arXiv:1807.01353v1.

\bibitem{DPSTT1} F. Dai, A. Prymak, A. Shadrin, V. Temlyakov, S. Tikhonov,
Sampling discretization of integral norms,
arXiv:2001.09320v1 [math.CA] 25 Jan 2020; Constructive Approximation, https://doi.org/10.1007/s00365-021-09539-0; Published online: 08 April 2021.

 \bibitem {DPSTT2} F. Dai, A. Prymak, A. Shadrin, V. Temlyakov, and S. Tikhonov, Entropy numbers and Marcinkiewicz-type discretization theorem, arXiv:2001.10636v1 [math.CA] 28 Jan 2020; J. Functional Analysis {\bf 281} (2021 109090.
 
 \bibitem{DW} F. Dai, Feng and  Heping Wang,  Positive cubature formulas and Marcinkiewicz-Zygmund inequalities on spherical caps, Constr. Approx., {\bf 31} (2010), 1--36.
 
\bibitem {FJS} T. Figiel, W.B. Johnson and G. Schechtman, Factorization of natural embeddings of $\ell_p^n$ into $L_r$, I, Studia Mathematica, {\bf 89} (1988),    79--103.

\bibitem{GKKW}   L. Gy{\"o}rfy, M. Kohler, A. Krzyzak, and H. Walk,
 A distribution-free theory of nonparametric regression, Springer,
 Berlin,  2002.

\bibitem{Gro} K. Gr{\"o}chenig, Sampling, Marcinkiewicz--Zygmund inequalities, approximation, and quadrature rules, {\it J. Approx. Theory}, {\bf 257} (2020).

\bibitem{JS} W.B. Johnson and G. Schechtman, Finite dimensional subspaces of $L_p$,
 	Handbook of the geometry of Banach spaces, Vol. 1 (2001), 837--870, North-Holland, Amsterdam.

\bibitem{Ka1}  B.S. Kashin,  On certain properties of the space of trigonometric
polynomials with the uniform norm,  Trudy Mat. Inst. Steklov,  {\bf 145}  (1980),  111--116;
 English transl. in  Proc. Steklov Inst. Math.,  {\bf 145}  (1981), 121-127.	
 
 \bibitem{KS} B.S. Kashin and A.A. Saakyan, Orthogonal Series, Providence, RI: American Mathematical Society, 1989.

\bibitem{KT3} B.S. Kashin and V.N. Temlyakov, On a norm and related applications, Mat. Zametki {\bf 64} (1998),  637--640.
 
\bibitem{KT4} B.S. Kashin and V.N. Temlyakov, On a norm and approximation characteristics of classes of functions of several variables,
Metric theory of functions and related problems in analysis, Izd. Nauchno-Issled. Aktuarno-Finans. Tsentra (AFTs), Moscow, 1999, 69--99.

\bibitem{KaTe03} B.S. Kashin and V.N. Temlyakov, The volume estimates and their applications,  East J. Approx., {\bf 9}  (2003), 469--485.

\bibitem{VT168} B.S. Kashin and V.N. Temlyakov, Observations on discretization of trigonometric polynomials with given spectrum, {\it Russian Math. Surveys}, {\bf 73:6} (2018), 1128--1130. Translation from {\it Uspekhi Mat. Nauk} {\bf 73:6} (2018) 197--198.

\bibitem{KKLT} B. Kashin, E. Kosov, I. Limonova, and V. Temlyakov, Sampling discretization and related problems, arXiv:2109.07567v1 [math.FA] 15 Sep 2021.

\bibitem{KUV} L. K{\"a}mmerer, T. Ullrich, and T. Volkmer, Worst-case recovery guarantees for least squares approximation using random samples,
Constructive Approximation, {\bf 54} (2021), 295?352.

\bibitem{KKP} J. Keiner, S. Kunis,  and D. Potts,  Efficient Reconstruction of Functions on the Sphere from Scattered Data, J Fourier Anal. Appl., {\bf 13}, (2007), 435--458. https://doi.org/10.1007/s00041-006-6915-y.

 \bibitem{Kos} E. Kosov, Marcinkiewicz-type discretization
of $L^p$-norms under the Nikolskii-type inequality assumption, arXiv:2005.01674v1 [math.FA] 4 May 2020.

\bibitem{KU} D. Krieg and M. Ullrich, Function values are enough for $L_2$-ap\-pro\-ximation, {\it Found. Comp. Math.}, doi:10.1007/s10208-020-09481-w;
arXiv:1905.02516v4 [math.NA] 19 Mar 2020.

\bibitem{KU2} D. Krieg and M. Ullrich, Function values are enough for $L_2$-approximation: Part II, {\it J. Complexity}, doi:10.1016/j.jco.2021.101569; arXiv:2011.01779v1 [math.NA] 3 Nov 2020.

\bibitem{LT} I. Limonova and V. Temlyakov, On sampling discretization in $L_2$, arXiv:2009.10789v1 [math.FA] 22 Sep 2020. 

\bibitem{MSS} A. Marcus, D.A. Spielman, and N. Srivastava,
Interlacing families II: Mixed characteristic polynomials and the Kadison-Singer problem,  {\it Annals of Math.}, {\bf 182:1} (2015), 327--350.

\bibitem{NSU} N. Nagel, M. Sch{\"a}fer, T. Ullrich, A new upper bound for sampling numbers, {\it Found. Comp. Math.}, Pub Date: 2021-04-26, DOI: 10.1007/s10208-021-09504-0;
arXiv:2010.00327v1 [math.NA] 30 Sep 2020. 

\bibitem{NOU} S. Nitzan, A. Olevskii, and A. Ulanovskii,
	{Exponential frames on unbounded sets},  {\it Proc. Amer. Math. Soc.}, {\bf 144:1} (2016), 109--118.
	
\bibitem{NoLN} E. Novak, {\em Deterministic and Stochastic Error Bounds in
Numerical Analysis}, Springer-Verlag, Berlin, 1988.	

\bibitem{PT} K. Pozharska and T. Ullrich, A note on sampling recovery of multivariate functions in the uniform norm, arXiv:2103.11124v2 [math.NA] 2 Apr 2021.

\bibitem{T93} V.N. Temlyakov, On approximate recovery of functions with bounded mixed derivative,  J. Complexity, {\bf 9} (1993), 41--59.

\bibitem{VTbook} V.N. Temlyakov, Greedy Approximation, Cambridge University
Press, 2011.

\bibitem{TT} V.N. Temlyakov and  S. Tikhonov, Remez-type and Nikol'skii-type inequalities:
general relations and the hyperbolic cross polynomials, Constr. Appr., {\bf 46} (2017),  593--615.

\bibitem{VTbookMA} V. Temlyakov, Multivariate Approximation, Cambridge University Press, 2018.

\bibitem{VT158} V.N. Temlyakov, The Marcinkewiecz-type discretization theorems for the hyperbolic cross polynomials, Jaen  Journal on Approximation, {\bf 9} (2017), No. 1, 37--63; arXiv: 1702.01617v2 [math.NA] 
26 May 2017. 

\bibitem{VT159} V.N. Temlyakov, The Marcinkiewicz-type discretization theorems, Constr. Approx. {\bf 48} (2018), 337--369; arXiv: 1703.03743v1 [math.NA] 10 Mar 2017.


\bibitem{VT183} V.N. Temlyakov, On optimal recovery in $L_2$, J. Complexity, {\bf 65}, (2021), 101545; Available online 19 December 2020, 101545;
doi: https://doi.org/10.1016/j.jco.2020.101545;
arXiv:2010.03103v1 [math.NA] 7 Oct 2020. 

\bibitem{TU1} V.N. Temlyakov and T. Ullrich, Bounds on Kolmogorov widths of classes with small mixed smoothness, J. Complexity, Available online 4 May 2021, 101575; arXiv:2012.09925v1 [math.NA] 17 Dec 2020.


 \bibitem{Z} A. Zygmund, Trigonometric Series, Cambridge University Press, 1959.


  \end{thebibliography}
\end{document}